\title{To the Hilbert class field from the hypergeometric modular function }
\author{Atsuhira Nagano and Hironori Shiga }
\documentclass{article}
\usepackage{mathrsfs,bm,graphics,amsfonts,amsthm,amssymb,amsmath,amscd}
\def\bigzerou{\smash{\lower1.7ex\hbox{\b 0}}}
\def\ds{\displaystyle}
\newtheorem{thm}{Theorem}[section]
\newtheorem{df}{Definition}[section]

\newtheorem{prop}{Proposition}[section]
\newtheorem{rem}{Remark}[section]

\newcommand{\Ima}{\rm Im\ }
\newcommand{\Mod}{\rm mod\ }
\def\comment#1{{ }}
\makeatletter
  
      \@addtoreset{equation}{section}
\begin{document}
\maketitle
\setlength{\baselineskip}{15 pt}

\begin{abstract}
 In this article we make an explicit approach to the problem:
 ``For a given CM field $M$, construct its maximal unramified abelian extension $C(M)$ by  the adjunction of special values of certain modular functions''
  in some restricted cases with $[M:\bm{Q}]\geq 4$.
 We make our argument based on Shimura's main result on the complex multiplication theory of his article in 1967. 
 His main result treats CM fields embedded in a quaternion algebra $\bm{B}$ over a totally real number field $F$. 
 We determine the modular function which gives the canonical model for all $\bm{B}$'s coming from arithmetic triangle groups.
 That is our main theorem. 
 As its application, we make an explicit case-study for $\bm{B}$ corresponding to the arithmetic triangle group $\Delta (3,3,5)$.  
By using the modular function of K. Koike obtained in 2003, we show several examples of the Hilbert class fields as an 
application of our theorem to this triangle group. 
 \end{abstract}
 
\section{Introduction}
In this article we make an explicit approach to the problem :
``For a given CM field $M$, construct its maximal unramified abelian extension $C(M)$ (i.e. the Hilbert
class field) by the adjunction of special values of certain modular functions.''

(a) The fundamental theorem of classical complex multiplication says that the Hilbert class
field of a given imaginary quadratic field $M = Q(\sqrt{-d})$ is given by the adjunction of the special 
value of the elliptic modular function 
$j(\omega), \omega = \begin{cases} \sqrt{-d} \ (d \equiv 1, 2 (\mod 4)) \cr (1+\sqrt{-d})/2\ (d \equiv 3 (\mod 4)),\end{cases}$ to $M$.
So we tried to obtain a counter part of this theorem in the case of CM fields of higher degree. That is also an important special case of the Hilbert 12th problem.
The theory of complex multiplication of abelian varieties was initiated by Shimura and Taniyama \cite{S-T} in 1950's. Shimura has worked for it about 10 years
(see \cite{SmrA}) . 
The main result on this theme is stated in \cite{SmrB}. 
We understand it as a theory of modular functions, in its refined sense, for the families of abelian varieties $A$ 
of {\rm QM}-type (namely a family of $A$ which extended algebra of endomorphisms ${\rm End}_0(A)$ contains a certain fixed quaternion 
algebra saying $\bm{B}$). If we have an embedding $\bm{B} \hookrightarrow (M_2(\bm{R}))^r$  for some $r$, 
it appears necessarily modular functions defined on $\bm{H}^r$, the $r$-times product of the upper half complex plane $\bm{H}$.
We make our argument based on Shimura's main result for the restricted case $r = 1$.
 Because, it is the only case which enables us to construct the exact Hilbert class field. 
 
On the other hand, for the family of abelian varieties (or Jacobi varieties) with a fixed generalized complex multiplication 
by a cyclotomic field, we often find its 
period domain becomes to be a hyperball. So, in this case we have a modular function on the hyperball. We expect it is connected with a certain type of hypergeometric differential equation. For the case of Appell's $F_1$, Terada and Deligne--Mostow determined the total list of modular functions on the 
hyperball \cite{Trd}, \cite{D-M}.
As an important back ground of the fundamental theorem of classical complex multiplication, we have the Schneider theorem:
``For $\tau \in \overline{\bm{Q}} \cap \bm{H}$, we have $j(\tau)\in \overline{\bm{Q}}$ if and only if $\tau$ is 
 imaginary quadratic.'' For the modular function of Terada and Deligne-Mostow, it holds the Schneider type theorem 
 \cite{Shg} \cite{S-W}. So we are expecting that there exists some kind of theory of complex multiplication for it. But, at this moment, 
 we do not have such a general theory for Appell's $F_1$ case.
Our present study can be considered as an intersection of the theory on the Appell's $F_1$ type modular functions and 
Shimura's complex multiplication theory.

(b) Our main results are summarized as:

(i) For the case the quaternion algebra $\bm{B}$ is coming from the arithmetic triangle unit group in the complex multiplication theory of Shimura, 
we determined the modular function that gives the generator of the Hilbert class field  $C(M)$ of a CM field $M$ which is embedded in $\bm{B}$. Our modular function is obtained as  the Schwarz inverse of the Gauss hypergeometric differential equation.

(ii) For one of the above cases, we described the modular function in terms of Riemann theta constants. It is obtained as a restriction of the 
Appell $F_1$ type modular function to a one-dimensional hyperplane section of the hyperball (namely 1-dimensional Shimura variety). 
The modular function is the one for the family of the Koike pentagonal curves $w^5 = z(z-1)(z-\lambda_1)(z-\lambda_2)$ with two parameters 
$\lambda_1, \lambda_2$.

(iii) By using the above theta constant representation we obtained several examples of Hilbert class fields by approximate calculation.

For it, our exposition proceeds under the following framework.

In Section 2 we recall the inverse of the Schwarz map for the Gauss hypergeometric differential equation and the classification of all arithmetic
triangle group by K. Takeuchi. In Section 3 we retell the main theorem of Shimura which we are concerned for the special case $r = 1$. Furthermore, we determine the modular function that gives the canonical model in the sense of Shimura for the quaternion algebra $\bm{B }$ coming from an arithmetic triangle unit group. That is our main theorem. In Section 4 we study the case $\bm{B}$ is coming from the triangle group $\Delta = \Delta (3, 3, 5)$. In this case $\Delta$ is he unit group
$\Gamma (\mathcal{O},1)$ of $\bm{B}$ in the sense of Shimura. We show an explicit theta representation of the modular function which is 
the canonical model based on the result of K. Koike \cite{Kkk} in 2003. In Section 5 by using our main theorem we show several examples of the Hilbert class fields of the {\rm CM} fields those are embedded in $\bm{B}$ of Section 4.

Also, our present study would suggest a new direction of the arithmetic application of the theory of $K3$ modular function via the works of S. Kondo \cite{Knd} 
and Artebani, Sarti, Taki \cite{A-S-T}. Because the 2 dimensional hyperball obtained as the period domain of the family of the Koike pentagonal curves can be considered at the same time as that of $K3$ surfaces with non-symplectic cyclic automorphism of order 5.
Our work is very much inspired by the pioneering study by J. Voight \cite{Vgt}.
\section{Hypergeometric modular functions and arithmetic triangle groups}
 \subsection{Hypergeometric modular function}
Let us consider the Gauss hypergeometric differential equation 
\begin{eqnarray} \label{eq: GHGE}
&&
E(a,b,c): \lambda (1-\lambda )f'' +(c-(a+b+1)\lambda )f'-abf=0
\end{eqnarray}
with real parameters $a,b,c$. 
It has regular singular points at $\lambda =0,1,\infty$. The exponents at singularities are given by 
the Riemann scheme 
\begin{eqnarray} \label{eq:GRscheme}
\left\{ \begin{matrix}0&1&\infty \cr 0&0&a\cr 1-c&c-a-b&b\end{matrix}
\right\}.
\end{eqnarray}
We always assume the condition
\begin{eqnarray} \label{eq:inversecd}
&&
(\ast )\qquad \begin{cases}|1-c| +|c-a-b|+|a-b| <1,
\cr \cr
p=1/|1-c| , q=1/|c-a-b|, r=1/|a-b| \in {\bf N}\cup \{ \infty\}.
\end{cases}
\end{eqnarray}

Set  $\eta_1(\lambda) ,\eta_2(\lambda)$ be a basis of the space of solutions of (\ref{eq: GHGE}).
The ratio $\eta_2/\eta_1$ determines a single valued analytic function on the lower complex half plane $\bm{H}_{-}$.
According to the condition ($\ast$ ), by choosing an adequate basis the image can be considered to be a 
hyperbolic triangle $\nabla (p,q,r)$ on the upper half plane $\bm{H}$ with angles $\ds{ \frac{\pi}{p}, \frac{\pi}{q}, \frac{\pi}{r}}$.

The multivalued analytic map $\mathcal{F}$ obtained as the analytic continuation of the map $\eta_2/\eta_1$ on $\bm{H}_{-}$ 
is called the Schwarz map of (\ref{eq: GHGE}). The image of $\mathcal{F}$ is obtained by the iteration of reflection procedure 
of $\nabla = \nabla (p,q,r)$. Due to the condition ($\ast$ ), the reflection images of $\nabla$ makes a tessellation of 
$\bm{H}$. In other words, the monodromy group of  (\ref{eq: GHGE}) is given as the totality of the iteration of the 
reflection procedures of even times. We call this group a triangle group $\Delta =\Delta (p,q,r)$. Note that the fundamental region 
of this monodromy group $\Delta$ is composed of $\nabla$ and its reflection $\nabla'$ with respect to one side of $\nabla$.

The inverse map $\phi (z)$ of $\mathcal{F}$ becomes to be a modular function defined on $\bm{H}$ with respect to 
the triangle group $\Delta (p,q,r)$. Let $z_1,z_2,z_3$ be the vertices of $\nabla$ obtained as $\mathcal{F}(0),
\mathcal{F}(1),\mathcal{F}(\infty )$ , respectively. 

\begin{df}
Let  (\ref{eq: GHGE}) be a Gauss hypergeometric differential equation with the condition  ($\ast$). 
We call the inverse $\phi (z)$ of the above Schwarz map a normalized {\bf{ hypergeometric modular function}}.
Note that its values at the vertices are fixed in the form : $\phi (z_1)=0,\phi (z_2)=1,\phi (z_3)=\infty $.
\end{df}

For the case $(p,q,r)=(\infty , \infty ,\infty )$ we have the angular parameters $(a,b,c)=(\frac{1}{2}, \frac{1}{2},1)$. 
By looking at the fundamental domain and its reflections, we know  the monodromy group is given by 
  $\langle \begin{pmatrix} 1&2\cr 0&1 \end{pmatrix} ,\begin{pmatrix} 1&0\cr 2&1 \end{pmatrix} \rangle ={\Gamma} (2)$.
Hence, the normalized hypergeometric modular function for $\Delta (\infty , \infty ,\infty )$ is the
Legendre function $\lambda (z)$ with $\lambda (i \infty )=0, \lambda (0)=1,\lambda (1)=\infty$.

For the case $(p,q,r)=(2,3,\infty)$, it appears the angular parameters $(a,b,c)=(\frac{1}{12}, \frac{5}{12},1)$.  
By comparing the corresponding fundamental region with that of $\Delta (\infty , \infty ,\infty )$, 
we know the monodromy group coincides with $PSL_2(\bm{Z})=\Gamma $.
So we obtain the elliptic modular function  $\ds{J(z) = \frac{4}{27}\frac{(1-\lambda +\lambda^2)^3}{\lambda^2(1-\lambda)^2}}$
 as the normalized hypergeometric modular function for $\Delta (2 , 3 ,\infty )$. It holds 
 $J(\frac{1-\sqrt{-3}}{2})=0, J(i)=1, J(i\infty )=\infty$.

\subsection{Arithmetic triangle groups and quaternion algebras}
Let $F$ be a totally real number field. Let $a,b$ be elements of $F$ satisfying the condition
\begin{eqnarray*}
\rm{(Cd)} :
\begin{cases}
a,b\in F: a<0,b>0 \cr \cr
 \ \mbox{all\ their\ conjugates\ other\ than} \ a,b\ \mbox{are\ negative.}
 \end{cases}
\end{eqnarray*}
By them we set a quaternion algebra
$\bm{B} =F+F\alpha +F\beta +F \alpha \beta \  ,\beta \alpha =-\alpha \beta ,
\alpha ^2=a, \beta^2=b$. 
We denote it by $\left( \frac{a,b}{F} \right)$.
For
$x=x_1+x_2\alpha +x_3\beta +x_4\alpha \beta \in \bm{B}$, we define its conjugate
$\overline{x}= x_1-x_2\alpha -x_3\beta -x_4\alpha \beta$. 

The reduced trace and the reduced norm of $x\in \bm{B}$ are 
defined by ${\rm{Trd}} (x) =x+\overline{x}, {\rm{Nrd}} (x)=x\overline{x}$, respectively.

Setting
\begin{eqnarray*}
&&
M_1=\begin{pmatrix} 1&0\cr 0&1\end{pmatrix},
M_x=\begin{pmatrix} 0&a\cr 1&0\end{pmatrix},
M_y=\begin{pmatrix} \sqrt{b}&0\cr 0&-\sqrt{b}\end{pmatrix},
M_z=M_xM_y,
\end{eqnarray*}
we have an isomorphism
\[
\bm{B} \cong F\ M_1+F\ M_x+F\ M_y+F\ M_z
\]
of non commutative $F$-algebras.
In this way $\bm{B}$ is embedded in $M_2(\bm{R})$.
We often identify $\bm{B}$ with the subalgebra $F\ M_1+F\ M_x+F\ M_y+F\ M_z \subset M_2(\bm{R})$ 
by this embedding.
So we have
${\rm{Trd}} (x) ={\rm{Tr}} (x) , {\rm{Nrd}} (x)=\det (x)$.
 
Set $\bm{B}^+=\{ x\in \bm{B} : \det (x)\gg 0\}$, where $x \gg 0$ means $x$ is totally positive as an element of $F$.
Due to the condition (Cd), $\bm{B}^+$ becomes to be a Fuchsian group acting on $\bm{H}$.
Let $\mathcal{O}_F$ denote the ring of integers of $F$. 
We denote the group of units in $\mathcal{O}_F$ by $E_1$, and set 
$E_0=\{ g\in E_1: \det (g) \gg 0\}$.
If it holds  ${\rm{Tr}} (\gamma)\in \mathcal{O}_F$ and $\det (\gamma)\in \mathcal{O}_F$ 
for an element $\gamma \in \bm{B}$, we say it is an integral element of $\bm{B}$.

If an  $\mathcal{O}_F$-module of rank 4 is a subring of $\bm{B}$, we say it is an order of $\bm{B}$. 
Note that any element of an order is an integral element. 
We fix one maximal order of $\bm{B}$ and denote it by $\mathcal{O}=\mathcal{O}_{\bm{B}}$.
\begin{rem}
Due to Takeuchi \cite{Tku2} (Prop. 3, p.206), we have the following.
For the quaternion algebra $\bm{B}$ coming from an arithmetic triangle group, $\bm{B}$ has unique 
maximal order up to conjugation by an element of $F$.
\end{rem}
\begin{df}
According to Takeuchi \cite{Tku2} (left) and Shimura \cite{SmrB} (right) we use the following notation: 
\begin{eqnarray} \label{eq:dfGammagrps}
&&
\Gamma^{(1)} (\bm{B}, \mathcal{O})= \Gamma (\mathcal{O})=\{ \gamma \in \mathcal{O}: \det (\gamma)=1 \},
\\&&
\Gamma^{+} (\bm{B}, \mathcal{O})=\Gamma (\mathcal{O},1)=\{ \gamma \in \mathcal{O}: \det (\gamma)\in E_0\},
\\&&
\Gamma^{(\ast)} (\bm{B}, \mathcal{O})=\Gamma^{\ast}   (\mathcal{O}) =\{ \gamma \in \bm{B}^+: \gamma \mathcal{O} =\mathcal{O} \gamma\}.
\end{eqnarray}
\end{df}

These groups are called conventionally as the norm 1 group, the unit group and the normalizer group respectively.
Note that the norm 1 group is a subgroup of the unit group of finite index, and 
the unit  group is a subgroup of the normalizer group of finite index. 
\begin{df}
If a triangle group $\Delta$ is commensurable with the unit group of a certain quaternion algebra 
$\ds{\bm{B} =\left( \frac{a,b}{F} \right)}$ up to conjugation in $SL_2(\bm{R})$, we say 
$\Delta$ is {\bf an arithmetic triangle group}.
\end{df}

K. Takeuchi  (see \cite{Tku1}, \cite{Tku2}) 
has determined all arithmetic triangle groups. There are 85 in total, and they are classified into 19 classes 
of commensurable families (see  Table App.1 in Appendix).

Moreover Takeuchi showed that 
the quaternion algebra $\ds{\bm{B} =\left( \frac{a,b}{F} \right)}$ corresponding to 
$\Delta =\Delta (e_1,e_2,e_3)$ $\quad (e_1\leq e_2\leq e_3)$ is given by
\begin{eqnarray*}
&&
a=t_2^2(t_2^2-4),
\\&&
b=t_2^2t_3^2(t_1^2+t_2^2+t_3^2+t_1t_2t_3-4),
\\&&
F=\bm{Q}(t_1^2,t_2^2,t_3^2,t_1t_2t_3), \mbox{where}\ t_i=2\cos \frac{\pi}{e_i}.
\end{eqnarray*}
\begin{rem} \label{rem2.1}
The class number $h(F)$ of $F$ is always equal to 1.
\end{rem}
\begin{rem} \label{excluderem}
We are interested in the unit group $\Gamma^{+}(\bm{B},\mathcal{O})$. In the following Table 2.1 there are two cases, 
(2) and (12), where it is not a triangle group but is a quadrangle group. We exclude these cases in the following argument.
\end{rem}

{\small
{\center
~
\vskip 3mm
\begin{tabular}{ccccccccc} \hline \hline
 &$F$& Disc.&$\Gamma^{(1)} (\bm{B}, \mathcal{O})$&$\Gamma^{+} (\bm{B}, \mathcal{O})$&$\Gamma^{(\ast)} (\bm{B}, \mathcal{O})$&$(a,b)$\\ \hline
(1)&$\bm{Q}$&$(1)$&$(2,3,\infty)$&$(2,3,\infty)$&$(2,3,\infty)$&$(-3,4)$  \\ 
(2)&$\bm{Q}$&$(2)(3)$&$(0;2,2,3,3)$&$(0;2,2,3,3)$&$(2,4,6)$&$(-4,6)$  \\
(3)&$\bm{Q}(\sqrt{2})$&$\frak{p}_2$&$(3,3,4)$&$(3,3,4)$&$(2,3,8)$&$(-3,4\cos^2(\frac{\pi}{8})(-3+4 \cos^2 (\frac{\pi}{8})))$  \\
(4)&$\bm{Q}(\sqrt{3})$&$\frak{p}_2$&$(3,3,6)$&$(2,3,12)$&$(2,3,12)$&$(-3,1+\sqrt{3})$  \\
(5)&$\bm{Q}(\sqrt{3})$&$\frak{p}_3$&$(0;2,2,2,6)$&$(2,4,12)$&$(2,4,12)$&$(-4,6+4\sqrt{3})$  \\
(6)&$\bm{Q}(\sqrt{5})$&$\frak{p}_2$&$(2,5,5)$&$(2,5,5)$&$(2,4,5)$&$(-4,1+\sqrt{5})$  \\
(7)&$\bm{Q}(\sqrt{5})$&$\frak{p}_3$&$(3,5,5)$&$(3,5,5)$&$(2,5,6)$&$(-\frac{1}{2}(5+\sqrt{5}),3(2+\sqrt{5}))$  \\
(8)&$\bm{Q}(\sqrt{5})$&$\frak{p}_5$&$(3,3,5)$&$(3,3,5)$&$(2,3,10)$&$(-3,\sqrt{5})$  \\
(9)&$\bm{Q}(\sqrt{6})$&$\frak{p}_2$&$(0;2,3,3,3)$&$(3,4,6)$&$(3,4,6)$&$(-4,6(2+\sqrt{6}))$  \\
(10)&$\bm{Q}(\cos (\frac{\pi}{7}))$&$(1)$&$(2,3,7)$&$(2,3,7)$&$(2,3,7)$&$(-3,4\cos^2(\frac{\pi}{7})(-3+4\cos^2(\frac{\pi}{7})))$  \\
(11)&$\bm{Q}(\cos (\frac{\pi}{9}))$&$(1)$&$(2,3,9)$&$(2,3,9)$&$(2,3,9)$&$(-3,4\cos^2(\frac{\pi}{9})(-3+4\cos^2(\frac{\pi}{9})))$  \\
(12)&$\bm{Q}(\cos (\frac{\pi}{9}))$&$\frak{p}_2\frak{p}_3$&$(0;2,2,9,9)$&$(0;2,2,9,9)$&$(2,4,18)$&$(-4,8\cos^2(\frac{\pi}{18})(-2+4\cos^2(\frac{\pi}{18})))$  \\
(13)&$\bm{Q}(\cos (\frac{\pi}{8}))$&$\frak{p}_2$&$(3,3,8)$&$(3,3,8)$&$(2,3,16)$&$(-3,4\cos^2(\frac{\pi}{16})(-3+4\cos^2(\frac{\pi}{16})))$  \\
(14)&$\bm{Q}(\cos (\frac{\pi}{10}))$&$\frak{p}_2$&$(5,5,10)$&$(2,5,20)$&$(2,5,20)$&\hspace{-3mm}{\scriptsize$\begin{pmatrix} -\frac{1}{2}(5+\sqrt{5}),\cr (3+\sqrt{5})\cos^2 (\frac{\pi}{20}(-5+\sqrt{5}+8\cos^2 (\frac{\pi}{20})))\end{pmatrix}$ } \\
(15)&$\bm{Q}(\cos (\frac{\pi}{12}))$&$\frak{p}_2$&$(3,3,12)$&$(2,3,24)$&$(2,3,24)$&{\footnotesize$\begin{pmatrix} -3,\cr 4\cos^2 (\frac{\pi}{24}(-3+4\cos^2 (\frac{\pi}{24})))\end{pmatrix}$}  \\
(16)&$\bm{Q}(\cos (\frac{\pi}{15}))$&$\frak{p}_3$&$(5,5,15)$&$(2,5,30)$&$(2,5,30)$&\hspace{-3mm}{\scriptsize$\begin{pmatrix} -\frac{1}{2}(5+\sqrt{5}),\cr (3+\sqrt{5})\cos^2 (\frac{\pi}{30}(-5+\sqrt{5}+8\cos^2 (\frac{\pi}{30})))\end{pmatrix}$ } \\
(17)&$\bm{Q}(\cos (\frac{\pi}{15}))$&$\frak{p}_5$&$(3,3,15)$&$(2,3,30)$&$(2,3,30)$&$(-3, 4\cos^2(\frac{\pi}{30})(-3+4\cos^2(\frac{\pi}{30})))$ \\
(18)&$\bm{Q}(\sqrt{2},\sqrt{5})$&$\frak{p}_2$&$(4,5,5)$&$(4,5,5)$&$(2,5,8)$& \hspace{-3mm} {\scriptsize$ \begin{pmatrix} -\frac{1}{2}(5+\sqrt{5}),\cr \frac{1}{8} (3+\sqrt{5})\cos^2 (\frac{\pi}{8}(-5+\sqrt{5}+\cos^2 (\frac{\pi}{8})))\end{pmatrix}$ } \\
(19)&$\bm{Q}(\cos (\frac{\pi}{11}))$&$(1)$&$(2,3,11)$&$(2,3,11)$&$(2,3,11)$&$(-3, 4\cos^2(\frac{\pi}{11})(-3+4\cos^2(\frac{\pi}{11})))$  \\
 \hline
\hline
\end{tabular}
\par \vskip 3mm
{\centerline{Table 2.1:List of norm 1 groups, unit groups and normalizer groups according to
K. Takeuchi \cite{Tku2} }
}
}
}
\vskip 5mm
%
%
%
%
\section{Description of Shimura's complex multiplication theorem for triangle cases}
Let $F$ be a totally real number field, and let $M$ be a CM field over $F$. 
Namely, $M$ is a totally imaginary field that is a quadratic extension of $F$.
Shimura made a long-ranged research on the complex multiplication theory of such CM fields. 
We can find his main result in \cite{SmrB}. He showed four main theorems there. 
Especially we are concerned with the first main theorem for a special case $r=1$. 
We can restate this specialized case as the following.

Let $F$ be a totally real number field, and let $\ds{\bm{B} =\left( \frac{a,b}{F} \right)}$  be a quaternion algebra 
satisfying the condition (Cd). In this case $\bm{B}$ satisfies the condition $r=1$ with the terminology of Shimura. 

\paragraph{[Shimura's  Main theorem I for the case $r=1$]}(\cite{SmrB} Theorem 3.2, p.73).
Take above mentioned $F,M$ and $\bm{B}$.
Assume an embedding $f: M \hookrightarrow \bm{B}$ satisfying $f(\mathcal{O}_M)\subset \mathcal{O}_{\bm{B}}$, 
where $\mathcal{O}_M$ stands for the ring of integers of $M$.
Then there are a nonsingular complex variety $V$ and a modular map $\psi (z)$ from $\bm{H}$ to $V$
with respect to $\Gamma (\mathcal{O},1)$
satisfying the following  condition:

(1) $\psi (z)$ induces a biholomorphic correspondence  $\bm{H}/\Gamma ({\mathcal O},1)\cong V$,

(2) $V$ is defined over $C(F)$ (the Hilbert class field of $F$), 

(3) for a regular fixed point (that is explained below)  $z_0\in \bm{H}$ of $M$, it holds
 $M(\psi (z_0))\cdot C(F) =C(M)$, where $C(M)$ stands for the Hilbert class field of $M$.
\par \vskip 3mm

\paragraph{[Regular fixed point  $z_0\in \bm{H}$ of $M$].} 
Recall the embedding  $f: M\hookrightarrow \bm{B}$. We can put $M=F(\alpha), \alpha \in \mathcal{O}_M$.
The linear transformation $g=f(\alpha)$ has unique fixed point in $\bm{H}$. 
That gives our regular fixed point.

\begin{df}
The above pair $(\psi ,V)$ is called {\bf a canonical model} for $\bm{H}/\Gamma ({\mathcal O},1)$.
\end{df}
\begin{rem} \label{ShimuraUniquenessThm}
The canonical model is unique up to ${\rm Aut}_{C(F)} (V)$.
(see \cite{SmrB} Theorem 3.3).
\end{rem}

We are going to give a visualization of the above canonical model theorem for triangle cases.

Suppose a quaternion algebra 
$\bm{B} =\left( \frac{a,b}{F} \right)$ is corresponding to a certain arithmetic triangle group. 
Set $\Delta (e_1,e_2,e_3)=\Gamma^{+}(\bm{B}, \mathcal{O}) (=\Gamma ({\mathcal O},1))$.
Suppose the triangle group $\Delta (e_1,e_2,e_3)$ is generated by the triangle $\nabla=\nabla (e_1,e_2,e_3)$.
And let $z_1,z_2,z_3$ be the corresponding vertices of $\nabla$, respectively.

Let us observe the shape of $\nabla$.
Due to Takeuchi's table 2.1 there are 16 commensurable classes for compact arithmetic triangle groups. 
As for the unit groups of the corresponding quaternion algebras, there are 10 scalene triangles as their 
generating $\nabla$'s and 6 $\nabla$'s isosceles triangles  any of which is not a regular triangle.
We call the former {\bf of scalene type} unit group and the latter {\bf of isosceles type} unit group.
 
\begin{prop}
Let $\bm{B} =\left( \frac{a,b}{F} \right)$ be a quaternion algebra corresponding to 
an isosceles type unit group $\Delta(e_1,e_1,e_3) \ (e_1=e_2\not= e_3)$ in Table 2.1.
Suppose a triangle $\nabla (z_1,z_2,z_3)$ is the fundamental triangle that generates $\Delta(e_1,e_1,e_3)$ and 
$z_i$ be the vertex which is the elliptic point of order $e_i$.

(i) The extension $M_0=F(\zeta_{e_1})$ becomes to be a CM field over $F$, where $\zeta_{\nu}$ means a 
primitive root of unity of order $\nu$.
Moreover, 
it holds $M_0=F(i \sqrt{\rho})$ with the following Table 3.1 of $\rho$.

(ii) The elliptic point $z_1$ and $z_2$ are regular fixed points of $M_0\subset \bm{B}$.

(iii) The class number $h(M_0)$ of $M_0$ is equal to 1.
\end{prop}

{\center

~
\vskip 3mm
\begin{tabular}{ccccccc} \hline
number&$F$&$\Delta (e_1,e_1,e_3)$&$M_0$&$\rho$&\\ \hline
(3)&$\bm{Q}(\sqrt{2})$&$\Delta (3,3,4)$&$F(\zeta_3)$&$3$&\\ \hline
(6)&$\bm{Q}(\sqrt{5})$&$\Delta (5,5,2)$&$F(\zeta_5)$&$\sin (\frac{\pi}{5})$&\\ \hline
(7)&$\bm{Q}(\sqrt{5})$&$\Delta (5,5,3)$&$F(\zeta_5)$&$\sin (\frac{\pi}{5})$&\\ \hline
(8)&$\bm{Q}(\sqrt{5})$&$\Delta (3,3,5)$&$F(\zeta_3)$&$3$&\\ \hline
(13)&$\bm{Q}(\cos (\frac{\pi}{8}))$&$\Delta (3,3,8)$&$F(\zeta_3)$&$3$&\\ \hline
(18)&$\bm{Q}(\sqrt{2}, \sqrt{5})$&$\Delta (5,5,4)$&$F(\zeta_5)$&$\sin (\frac{\pi}{5})$&\\ \hline

\hline
\end{tabular}
\vskip 5mm
\centerline{Table 3.1:Generators of CM field $M_0$ for isosceles types}
}
\par \vskip 4mm
\begin{proof}
(i) We can obtain the assertion by direct observations.

(ii) Let $\gamma_i$ be a generator of the isotropy group of  $z_i (i=1,2,3)$ in $M_2(\bm{R})$. Then it holds 
$\Delta (e_1,e_2,e_3)=\langle \gamma_1,\gamma_2,\gamma_3 \rangle $ and $\gamma_i^{e_i}={\rm id}$.
So, it holds $M_0\cong F(\gamma_1) (=F(\gamma_2))$. Hence, $z_1(z_2, resp.)$ is a regular fixed point of $M_0$.

(iii) We have a well known formula for the class number of the biquadratic field of the type $\bm{Q} (\sqrt{5}, \sqrt{-\ell})$. 
Moreover, we can refer the table of Yamamura \cite{Ymm} to know the class number of other types of  $M_0$.
\end{proof}


\begin{thm} {\rm (Main Theorem)}
Let $\ds{\bm{B}/F }$ a quaternion algebra coming from a compact type arithmetic triangle group. 
Set  $\Delta =\Delta (e_1,e_2,e_3)=\Gamma ({\mathcal O}_{\bm{B}},1)$.
Let  $\nabla (z_1,z_2,z_3)$ be a triangle on $\bm{H}$ with vertices $z_i (i=1,2,3)$ which generates 
the triangle group $\Delta (e_1,e_2,e_3)$. Assume the order of $z_i$ is equal to $e_i (i=1,2,3)$.

(I) The case $\Delta$ is an scalene type unit group.  Let $\varphi (z)$ be a hypergeometric modular function 
with respect to $\Delta (e_1,e_2,e_3)$ which is normalized with the condition

\[
{\rm (Ncd):} \qquad  \varphi (z_1)=1, \varphi (z_2)=-1, \varphi (z_3)=\infty.
\]
Then it gives the canonical model of $\bm{H}/\Gamma ({\mathcal O}_{\bm{B}},1)$ together with the 
Riemann sphere $S=\bm{P}^1$ that is the image of $\varphi (z)$.

(II)  The case $\Delta$ is an isosceles type unit group. 
Let $\varphi (z)$ be a hypergeometric modular function with the same condition (Ncd) in (I). 
We make another function $\tilde{\varphi} (z)$ by using a pure imaginary number $i\sqrt{\rho}$,
 where $\rho$ is the number obtained in the previous proposition:
\[
\tilde{\varphi} (z) =i\sqrt{\rho} \cdot \varphi (z),
\]
Then one of $\varphi (z)$ and $\tilde{\varphi} (z)$ becomes to be the canonical model together with 
the image $S=\bm{P}^1$.
 \end{thm}
 
 \begin{rem}
 At this moment we don't have a criterion to determine which of two candidates in (II) becomes to be 
 the canonical model. Later, we shall show the example that $\tilde{\varphi}(z)$ is the canonical model.
 \end{rem}
 \begin{proof}
 
 Let $(\psi, V)$ be a canonical model that is assured the existence by the above Shimura's Main theorem.
 Because $\Gamma (\mathcal{O}_{\bm{B}},1)$ is a triangle group, $V$ is of genus $0$ as a 
 complex algebraic curve over $F$. 
 Let $q$ be a holomorphic isomorphism (namely, a morphism over $\bm{C}$) from $V$ to $S$ defined by the property $\varphi =q\circ \psi$.
 
Recalling Remark \ref{rem2.1},
if we can show that $q$ is a morphism defined over $F$, by observing the above Shimura's theorem and Remark 
\ref{ShimuraUniquenessThm}
$(\varphi ,S)$ becomes to be a canonical model.
 We may assume that $V$ is 
 given as a projective variety in $\bm{P}^N$ defined by a homogeneous ideal in 
 $F[X_0,\cdots, X_N]$, where $[X_0,\cdots ,X_N]$ being homogeneous coordinates of $\bm{P}^N$.
 For any element $\tau \in {\rm Gal} (\bm{C}/F)$, we have an action $\tau : V\rightarrow V$ by 
 $\underline{X}=(X_0,\ldots ,X_N) \mapsto \underline{X}^{\tau}=(X_0^{\tau},\ldots ,X_N^{\tau})$.
 
 So, $q$ is given by the form
 $q(X_0,\ldots ,X_N)=[P_0(X_0,\ldots ,X_N),P_1(X_0,\ldots ,X_N)] \in S$:
 \begin{eqnarray*}
 &&
 \begin{cases}
 P_0(X_0,\ldots ,X_N)= \sum_{\nu}  c_{0\nu} X^{\nu} , \cr
  P_1(X_0,\ldots ,X_N)= \sum_{\nu}  c_{1\nu} X^{\nu}
 \end{cases}
 \end{eqnarray*}
  with some $P_0,P_1\in \bm{C} [X_0,\ldots ,X_N]$ of the same degree.
 Hence, we can define a $\bm{C}$-morphism $q^{\tau} : V\rightarrow S$ by
 \begin{eqnarray*}
 &&
 q^{\tau} =[P_0^{\tau}(X_0,\ldots ,X_N),P_1^{\tau}(X_0,\ldots ,X_N)] ,
 \end{eqnarray*}
  with
 \[
 P_i^{\tau} (X_0,\ldots ,X_N)= \sum_{\nu}  c_{i\nu}^{\tau} X^{\nu}, \ (i=1,2).
 \]
 Note that
 \[
 q^{\tau}  (X_0,\ldots ,X_N)=( q(\tau^{-1}(X_0, \ldots ,X_N)))^{\tau}.
 \]
 We have the following diagram.
 \begin{center}
\begin{picture}(250,140)
   \put(147,120){$S$}
 \put(147,50){$V$}
   \put(60,120){$\bm{H}$}
    \put(85,90){$\bm{H}/\Gamma (\mathcal{O},1)$}
   \put(70,115){\vector(1,-1){15}}
    \put(127,100){\vector(1,1){15}}
     \put(73,120){\vector(1,0){70}}
 \put(60,50){$\bm{H}$}
  \put(85,20){$\bm{H}/\Gamma (\mathcal{O},1)$}
   \put(70,45){\vector(1,-1){15}}
    \put(127,30){\vector(1,1){15}}
     \put(73,50){\vector(1,0){70}}
      \put(110,55){$\psi$}
       \put(110,125){$\varphi$}
     \put(66,62){\line(0,1){53}}
       \put(64,62){\line(0,1){53}}
     \put(150,62){\vector(0,1){53}}
      \put(155,90){$q$}
        \put(104,30){\line(0,1){53}}
       \put(106,30){\line(0,1){53}}
  \end{picture}
\vspace*{-10pt}

Diagram 3.1  Two modular functions $\varphi$ and $\psi$

\end{center}

Take any element $\tau \in {\rm Gal} (\bm{C}/F)$.
By observing Table 2.1, we assume $e_3\not= e_1$ and $e_3\not= e_2$.

According to Theorem 3.17  in \cite{SmrB} (also we should use the condition $C(F)=F$ in our case ) for any index $i$ we have 
\begin{eqnarray} \label{eq:Shimura317Coseq}
&&
(\psi (z_i))^{\tau^{-1}}=\psi (w_i)
\end{eqnarray}
 with some elliptic point $w_i$ of $\Gamma (\mathcal{O},1)$ which has the same order as $z_i$.
 In case $i=3$, the corresponding $w_i=w_3$ should be an elliptic point that is $\Gamma (\mathcal{O},1)$-equivalent to $z_3$. 
 So we have 
 \begin{eqnarray} \label{eq:psiz3}
 &&
 \psi (z_3)^{\tau}=\psi (z_3).
 \end{eqnarray}
Namely, $\psi (z_3)$ is a $F$-rational point of $V$.

Since $V$ is a curve of genus $0$ defined over $F$, the image of the anti-canonical map is a plane 
conic defined over $F$.
In our case, $V$ has at least one $F$-rational point, saying $v_{\infty}$. 
So the projection from $V$ to a 
coordinate plane centered at $v_{\infty}$ defines an $F$-isomorphism from $V$ to $\bm{P}^1$. 
Consequently, we may assume $V=\bm{P}^1$ from the beginning. So we may put
\[
q(v)=\frac{a v+b}{cv+d},
\]
with $\begin{pmatrix} a&b\cr c&d \end{pmatrix} \in GL_2(\bm{C})$. 
We set
$\ds{ q^{\tau} (v)=\frac{a^{\tau} v+b^{\tau}}{c^{\tau}v+d^{\tau}} }$.
 So we have
 \begin{eqnarray*}
 &&
 q^{\tau} (\psi (z_i)) =[q (\tau^{-1}(\psi (z_i)))]^{\tau} =
[q (\psi (w_i))]^{\tau}  =\varphi (w_i)^{\tau} =\varphi (w_i) =q(\psi (w_i)).\qquad (\ast)
 \end{eqnarray*}
 Here, the second equality is deduced from ( \ref{eq:Shimura317Coseq}).

 For $z_1$ and $z_2$, it remains two possibilities:
 \begin{eqnarray*}
 &&
 \begin{cases}
 {\rm (a)}  \ q^{\tau}(\psi (z_1))=q(\psi (z_1)) \ \mbox{and} \ q^{\tau}(\psi (z_2))=q(\psi (z_2)),
 \ \mbox{for\ any\ element \ $\tau$\ of} \ {\rm Gal} (\bm{C}/F)
 \cr
  {\rm (b)} \  q^{\tau}(\psi (z_1))=q(\psi (z_2))\ \mbox{and} \  q^{\tau}(\psi (z_2))=q(\psi (z_1)),
   \ \mbox{for\ some\ element \ $\tau = \tau_0$\ of} \ {\rm Gal} (\bm{C}/F).
 \end{cases}
 \end{eqnarray*}
 We make our argument by separating several cases.
 
 (i) The case of scalene type unit group. 
 
 By the same argument as for $z_3$, we obtain that $\psi (w_i)=\psi (z_i)$ and that $v_i=\psi (z_i)\in F,  \ (i=1,2,3)$.
 Observing the equality $(\ast)$, we have $q^{\tau} (v_i)=q (v_i)\ (i=1,2,3)$. Namely, it holds
 \[
 \frac{a^{\tau} v_i+b^{\tau}}{c^{\tau} v_i+d^{\tau}} =\frac{av_i+b}{cv_i+d} \ (i=1,2,3)
 \]
 for every $\tau \in {\rm Gal} (\bm{C}/F)$. It means that $q$ is defined over $F$, namely $q$ is a $F$ isomorphism. 
 Referring Remark 3.1, we know that $(\varphi , S)$ is a canonical model.
 
  (ii) The case of isosceles type unit group. 
  
  If it holds above (a), we obtain the canonical model $(\varphi ,S)$ by the same argument as in (i). 
  So we concentrate our argument to the case (b).
 
 (step 1) We may assume $\varphi (z_3)=\psi (z_3)=\infty$. So $q$ takes the form 
 $q(v)=a v+b, q^{-1} (t) =\alpha t+\beta$.
 
 (step 2) It holds
 \[
 \begin{cases}
 \psi (z_1)+\psi (z_2)=q^{-1} (\varphi (z_1))+q^{-1} (\varphi (z_2))=
 \alpha (\varphi (z_1)+\varphi (z_2))+2\beta =2\beta
 \cr
  ={q^{-1}}^{\tau} (\varphi (z_2))+{q^{-1}}^{\tau} (\varphi (z_1)) 
  =\alpha^{\tau}  (\varphi (z_1)+\varphi (z_2))+2\beta^{\tau} =2\beta^{\tau}
 \end{cases}
 \]
 for any $\tau \in {\rm Gal} (\bm{C}/F)$. 
 So, $\beta \in F$. Hence we may assume  $\beta =0$ by the $F$-isomorphism  $\psi \mapsto \psi -\beta/2$.
 
 (step 3) It holds
 \begin{eqnarray*}
 &&
 (\psi (z_1) \psi (z_2))^{\tau} = (\psi (z_1))^{\tau} ( \psi (z_2)))^{\tau} 
 =(q^{-1})^{\tau} (\varphi (z_1)^{\tau}) (q^{-1})^{\tau} (\varphi (z_2)^{\tau}) 
 \\&&
 =\alpha^{\tau} (q^{-1})^{\tau} (\varphi (z_1)^{\tau} \varphi (z_2)^{\tau} )
 =\alpha ^{\tau} (q^{-1})^{\tau} ((\varphi (z_1) \varphi (z_2))^{\tau} )
 \\&&
 \mbox{because $\varphi (z_1) \varphi (z_2) \in F$ , hence}
 \\&&
 =\alpha ^{\tau} (q^{-1})^{\tau} (\varphi (z_1) \varphi (z_2) )
 =(q^{-1})^{\tau} (\varphi (z_1))(q^{-1})^{\tau} (\varphi (z_2))
 =(q^{-1}) (\varphi (z_1))(q^{-1}) (\varphi (z_2))=\psi (z_1)\psi (z_2).
 \end{eqnarray*}
 Then,  $\psi (z_1)\psi (z_2)\in F$.
 
 (step 4) $\alpha ^{\tau} =-\alpha$ for $\tau =\tau_0$, especially $\alpha \notin F$.
 
 For,
 \begin{eqnarray*}
 &&
 (q^{-1}) (\varphi (z_1))(q^{-1}) (\varphi (z_2)) =\alpha^2 \varphi (z_1)\varphi (z_2)
\\&&
=  (q^{-1})^{\tau} (\varphi (z_1))(q^{-1})^{\tau} (\varphi (z_2)) =(\alpha ^{\tau} )^2 \varphi (z_1)\varphi (z_2).
\end{eqnarray*}

 (step 5) We may put $\tau_0$ to be the complex conjugation map $c \mapsto \overline{c}$. Especially
 $\alpha$ is a pure imaginary number.
 
 Unless, taking $\overline{\alpha} = \alpha ^\tau $ we have 
 \begin{eqnarray*}
 &&
\overline{\alpha} = \alpha ^\tau =\alpha ^\tau \varphi (z_1)
=(q^{-1})^{\tau} (\varphi (z_1))=(q^{-1}) (\varphi (z_1))=\alpha  \varphi (z_1)=\alpha.
\end{eqnarray*}
It means that $\alpha =\psi (z_1)$ is a real number. Considering Shimura's main theorem, it contradicts that $\psi (z_1) \notin F$ is a generator 
of the field $C(M_0)=M_0=F(i\sqrt{\rho})$. So we have the assertion.
 
 (step 6)  Up to a constant factor in $F$, it holds $\alpha =\psi (z_1)=i\ \sqrt{\rho}$ .

 For,  it holds $\alpha =\psi (z_1)\notin F$. 
 Due to Shimura's Main theorem, the singular value of the canonical model $\psi (z)$ at $z_1$ generates the Hilbert class field of $M_0$. 
  Observing the fact  $h(M_0)=1$, it must hold
   $\psi (z_1)\in F(i\sqrt{\rho})$. Because $\alpha$ is pure imaginary. We may put
 $\psi (z_1)=i\sqrt{\rho}$  up to a factor in $F$.
   
 (step 7) As a direct consequence of (step 6), In the case (b), it holds $\psi (z_1)\notin F$ and $\psi (z)= \tilde{\varphi}$.
 \end{proof}
 \begin{rem}
 As we mentioned in Remark \ref{excluderem}, we had to exclude two cases, (2) and (12), from our consideration. 
 We believe that even in these cases there would exists an explicit description of the canonical model in terms of 
 the hypergeometric modular function. Still note that, due to Shimura's main theorem, we obtain only the Hilbert class fields 
 of imaginary quadratic fields for (2). So, the case (12) remains to be interesting.
 \end{rem}
\section{An explicit form of Shimura's canonical model for $\Delta (3,3,5)$}
\paragraph{Koike's modular function for a family of Pentagonal curves.}
The Schwarz map for the Gauss hypergeometric differential equation $E(\frac{2}{5}, \frac{3}{5}, \frac{6}{5})$ can be 
identified (by way of the integral representation of the two independent solutions) with the period map
for a family of algebraic curves of genus 4:
\[
C(\lambda ): y^5=x^2(x-1)(x-\lambda)
\]
with a parameter $\lambda \ (\not= 0,1)$. The monodromy group is the triangle group $\Delta =\Delta (5,5,5)$ .
By considering the inverse map we obtain the modular function $\lambda (u)$ with respect to $\Delta$ 
that is defined on the period domain.

K. Koike \cite{Kkk} showed a representation of $\lambda (u)$ in terms of the Riemann theta constants.

Suppose $0<\lambda <1$. We regard $C(\lambda)$ as a 5-sheeted branched cover over the $x$-plane with cut lines 
connecting the base point $x_0\in \bm{H}_{-}$ and critical points (i.e. the projection of ramification points) 
 $x=0,\lambda ,1,\infty $. 
Set $\gamma_2, \gamma_3$ be two 1- cycles in $H_1(C(\lambda ),\bm{Z})$ indicated in Fig. 4.1  below, where 
(1) means the analytic continuation of the real branch of $y$ on $x>1$ along an arc that does not intersect the indicated cut lines, and 
(2) and (3) means the branches of $y$ given by  the images of (1) by the automorphisms $(x,y)\mapsto (x,\zeta_5y)$ and 
$(x,y)\mapsto (x,\zeta_5^2y)$, respectively, where $\zeta_5=e^{2\pi i/5}$.

  \vskip 5mm
{\center
\scalebox{.85}
{ \includegraphics{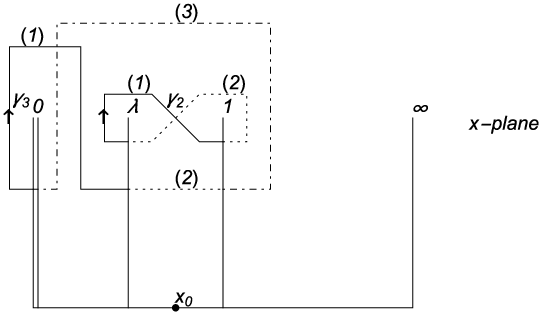} } 

 {Fig.4.1  : homology cycles on $C(\lambda)$ \\
 }
 }
\vskip 4mm

By making the analytic continuation, the integrals 
\[
\eta_2 (\lambda) =\int_{\gamma_2} \frac{dx}{y^2}, \eta_3 (\lambda) =\int_{\gamma_3} \frac{dx}{y^2}
\]
 give multivalued analytic functions
on the $\lambda$ space. They are  independent solutions of
 $E(\frac{2}{5}, \frac{3}{5}, \frac{6}{5})$.
According to Koike \cite{Kkk} , the image of the Schwarz map $\ds{\mathcal{F} (\lambda) = \frac{\eta_2(\lambda)}{\eta_3(\lambda)}}$
is given by the disc in $\bm{P}^1$:
\begin{eqnarray}
&&
\mathcal{D} =\{ [\eta_2,\eta_3]\in \bm{P}^1: |\eta_2|^2+\omega |\eta_3|^2<0\}, \ (\omega = (1-\sqrt{5})/2).
\end{eqnarray}
We have a modular embedding of $\mathcal{D}$ into the Siegel upper half space
 \[
 \frak{S}_4 =\{ \Omega \in GL_4(\bm{C}): ^t\Omega =\Omega , \Ima{\Omega} >0\}
 \] 
 by the following manner:
 {\small
 \begin{eqnarray*}
 &&
\Omega (u)=\frac{1}{ \eta_2^2-\left(1+e^{\frac{2 i \pi }{5}}\right) e^{-\frac{4 i \pi }{5}} \eta_3^2}
 \\&&
[ \left(
\begin{array}{cccc}
 \left(-1+e^{-\frac{4 i \pi }{5}}\right) \left( \eta_2^2+\eta_3^2\right) &
   \left(1-e^{\frac{4 i \pi }{5}}\right) \eta_2\eta_3& 0 & 0 \\
 \left(1-e^{\frac{4 i \pi }{5}}\right) \eta_2 \eta_3 & \left(-1+e^{\frac{4 i \pi
   }{5}}\right) \left(\eta_2^2-e^{-\frac{4 i \pi }{5}} \eta_3^2\right) & 0 & 0 \\
 e^{-\frac{4 i \pi }{5}} \left(\left(1+e^{\frac{2 i \pi }{5}}\right)
   \eta_2^2+\eta_3^2\right) & \left(1-e^{-\frac{4 i \pi }{5}}\right) \eta_2
   \eta_3 & 0 & 0 \\
 \left(-e^{\frac{2 i \pi }{5}}+e^{-\frac{4 i \pi }{5}}\right) \eta_2 \eta_3 &
   \left(e^{\frac{2 i \pi }{5}}+e^{\frac{4 i \pi }{5}}\right)
   \left(\eta_2^2-e^{-\frac{4 i \pi }{5}} \left(1+e^{\frac{4 i \pi }{5}}\right)
  \eta_3^2\right) & 0 & 0 \\
\end{array}
\right)
\\&&
+\left(
\begin{array}{cccc}
 0 & 0 & e^{-\frac{4 i \pi }{5}} \left(\left(1+e^{\frac{2 i \pi }{5}}\right)
   \eta_2^2+\eta_3^2\right) & \left(-e^{\frac{2 i \pi }{5}}+e^{-\frac{4 i \pi
   }{5}}\right) \eta_2 \eta_3 \\
 0 & 0 & \left(1-e^{-\frac{4 i \pi }{5}}\right) \eta_2 \eta_3 & \left(e^{\frac{2 i
   \pi }{5}}+e^{\frac{4 i \pi }{5}}\right) \left(\eta_2^2-e^{-\frac{4 i \pi }{5}}
   \left(1+e^{\frac{4 i \pi }{5}}\right) \eta_3^2\right) \\
 0 & 0 & -e^{\frac{4 i \pi }{5}} \left(\eta_2^2-\left(1+e^{\frac{2 i \pi }{5}}\right)
   \eta_3^2\right) & \left(e^{-\frac{2 i \pi }{5}}-e^{\frac{2 i \pi }{5}}\right)
  \eta_2\eta_3 \\
 0 & 0 & \left(e^{-\frac{2 i \pi }{5}}-e^{\frac{2 i \pi }{5}}\right) \eta_2 \eta_3
   & -e^{-\frac{4 i \pi }{5}} \left(\eta_2-\left(1+e^{-\frac{2 i \pi }{5}}\right)
   \eta_3^2\right) \\
\end{array}
\right)
],
 \end{eqnarray*}
 }
where $\ds{u=\frac{\eta_2}{\eta_3}}$.  
Note that the monodromy group $\Delta (5,5,5)$ is given as a subgroup of  $Sp_8(\bm{Z})$
that preserves $\Omega (\mathcal{D})$.
 
Set the Riemann theta constant on $\frak{S}_4$ with a characteristic  $(a,b) \in (\bm{Q}^4)^2$:
 \begin{eqnarray*}
 &&
 \vartheta \left[ \begin{matrix} a\cr b\end{matrix} \right] (\Omega)
 =
 \sum_{n\in \bm{Z}^4} \exp [ \pi i\ \! ^t (n+a) \Omega (n+a) +2\pi i \ \! ^t(n+a)b].
 \end{eqnarray*}
We define two following theta characteristics:
 \[
a_{11}=\left[ \begin{matrix} a\cr b\end{matrix} \right]  =\frac{1}{10} \left[ \begin{matrix}1&1&1&1 \cr -2&-2&-1&-1\end{matrix} \right] 
 \]
 and
  \[
a_{19}=\left[ \begin{matrix} a\cr b\end{matrix} \right]  =\frac{1}{10} \left[ \begin{matrix}1&9&1&9 \cr -2&-8&-1&-9\end{matrix} \right] .
 \]
 We define two theta functions on $\mathcal{D}$:
 \[
 \theta_{11} (u)=\vartheta [a_{11}] (\Omega (u)), \theta_{19} (u)=\vartheta [a_{19}] (\Omega (u)).
 \]
 \begin{thm} {\rm (K. Koike \cite{Kkk})}  The function $\lambda (u) = \left( \frac{\theta_{11} (u)}{\theta_{19} (u)}\right)^5$ on $\mathcal{D}$ gives the 
 inverse of the Schwarz map $\mathcal{F} (\lambda)$, and it gives a holomorphic isomorphism 
$\mathcal{D}/\Delta (5,5,5) \xrightarrow{\sim} \bm{P}^1(\bm{C})$.
Especially, it holds 
$\lambda (\omega e^{-2\pi i/5})=0, \lambda (\omega)=\infty  , \lambda (0)=1$, where $\omega =\frac{1-\sqrt{5}}{2}$.
\end{thm}
\begin{rem}
Set $(\xi_1,\xi_2,\xi_3)=(\omega  e^{-2\pi i/5},\omega,0)$. 
 The triangle $\mathcal{F} (\bm{H}_{-}) =\nabla (\xi_1, \xi_2,\xi_3)$ is a generating triangle of $\Delta (5,5,5)$.
\end{rem}
%
%
\section{Examples of the Hilbert class fields of higer degree}
\subsection{The normalized hypergeometric function for the case (8).}
Let us consider the quaternion algebra $\ds{ \bm{B} =\left( \frac{a,b}{F} \right) }$ arising in the class (VIII) in Table App 1 due to Takeuchi.
In this case, we have the unit group $\Gamma (\mathcal{O},1) =\Delta (3,3,5)$ (see Table 2.1).
We may regard $\Delta (5,5,5)$ is a subgroup of $\Delta (3,3,5)$ of index 3. Set $u_1$ be the center of gravity of $\nabla (\zeta_1,\zeta_2, \zeta_3)$, 
and set $(u_1,u_2,u_3)=(u_1,\overline{u_1}, 0)$.
The triangle $\nabla (u_1,u_2,u_3)$ becomes to be a generating triangle of $\Delta (3,3,5)$ (see Fig. 5.1)

By using Koike's $\lambda$ function, set
 \[
\Phi (u) = \frac{1}{3} \frac{\lambda ^3-3\lambda +1}{\lambda (\lambda -1)}.
\]
Then $\varphi (u) =\frac{2}{\sqrt{-3}} \left( \Phi (u) -\frac{1}{2} \right)$ 
is the normalized hypergeometric modular function for $\Delta (3,3,5)$  in the sense that 
$(\varphi (u_1), \varphi (u_2), \varphi (u_3))=(1,-1,\infty)$.
So, $\tilde{\varphi} (u) =\Phi (u) -\frac{1}{2} $ gives the modular function of the same symbol in the main theorem 3.1.

 \vskip 5mm
{\center
\scalebox{.9}
{ \includegraphics{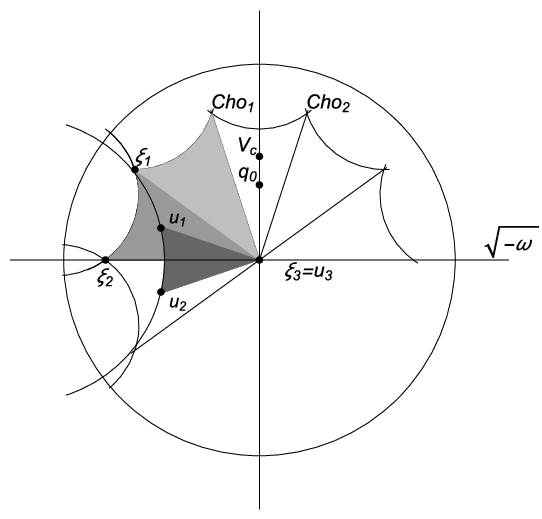} }

 Fig. 5.1.  : Relation between $\Delta (5,5,5)$ and $\Delta (3,3,5)$ \\
 }
 
\vskip 4mm
\begin{prop}
Set
\begin{eqnarray*}
&&
BG_1 = M_1,
\\&&
BG_2 = M_1+(\frac{1}{2} \omega)M_y +\frac{1}{2} M_z,
\\&&
BG_3 = (\frac{1}{2} - \frac{1}{2} \omega) M_1 +(\frac{1}{2} \omega) M_x,
\\&&
BG_4 = M_1+\omega M_y, \  (\omega =\frac{1-\sqrt{5}}{2}).
\end{eqnarray*}
They give a basis of the maximal order  $\mathcal{O}_{\bm{B}}$ as a $F$-module.
\end{prop}
\begin{proof}
We have the well known criterion for the basis of the maximal order (For example, refer \cite{VgtL} Prop. 4.8, p.50).

[Fact] Let $\bm{B}$ a quaternion algebra over $F$, and let $\{ G_1,G_2,G_3,G_4\}$ be a system of 
elements of $\bm{B}$. Then $\ds {\sum_{i=1}^4 FG_i}$ becomes to be  a maximal order of $\bm{B}$ if and only if 
the following two conditions are satisfied.

(i) The system $\{ G_1,G_2,G_3,G_4\}$ becomes to be a basis of an $F$-vector space, 
and $\ds {\sum_{i=1}^4 FG_i}$ becomes to be an order in $\bm{B}$.

(ii) $ \vert {\rm det }({\rm tr} (G_iG_j))_{1\leq i,j\leq 4} ) \vert = D(\bm{B})^2$ where $D(\bm{B})$ 
is the discriminant of $\bm{B}$.

We can check these conditions by explicit calculation.
\end{proof}
\begin{rem} \label{DHtransformRem}
Set
 \begin{eqnarray*}
 &&
  M_{mc} = \begin{pmatrix} \sqrt{\sqrt{5} \overline{\omega}}&0 \cr 0&1 \end{pmatrix},
  M_{mr} = \begin{pmatrix} i& i\sqrt{-\omega} \cr -1&\sqrt{-\omega} \end{pmatrix}.
 \end{eqnarray*}
 The composition $M_{hd}= M_{mc} M_{mr}$ induces an isomorphism $\mathcal{D} \xrightarrow{\sim}  \bm{H}$.
 For a linear transformation $h$ acting on $\mathcal{D}$,  we obtain a transformation  $\tilde{h} = M_{hd} \circ h\circ M_{hd}^{-1}$  acting on $\bm{H}$. 
 This shifting procedure induces an identification of  the triangle group $\Delta (3,3,5)$ and the unit group $\Gamma (\mathcal{O},1)$. 
 So,we may identify the modular function on $\bm{H}$ with that on $\mathcal{D}$ through this correspondence.
  \end{rem}
 \subsection{A generator systems of the triangle groups via the maximal order.}
 Set
 \begin{eqnarray*}
 &&
 g_{34}=\begin{pmatrix}e^{\pi i/10}& 0\cr 0&e^{-\pi i/10}\end{pmatrix},
 \\&&
 h_{34}=g_{34}^2 =\begin{pmatrix}e^{\pi i/5}& 0\cr 0&e^{-\pi i/5}\end{pmatrix},
 \\&&
h_{45}=
\left(
\begin{array}{cc}
 \frac{-6+6 \sqrt{5}+i \sqrt{50-10 \sqrt{5}}+i \sqrt{10-2 \sqrt{5}}}{4
   \left(-3+\sqrt{5}\right)} & \frac{2}{-3+\sqrt{5}} \\
 -\frac{i \left(-4 i+\sqrt{50-10 \sqrt{5}}+3 \sqrt{10-2 \sqrt{5}}\right)}{4
   \left(-3+\sqrt{5}\right)} & \frac{-7+\sqrt{5}-i \sqrt{10-2 \sqrt{5}}}{2
   \left(-3+\sqrt{5}\right)} \\
\end{array}
\right),\\&&
hn_{45} = e^{-\pi i/5} h_{45},
\\&&
h_{312}=\left(
\begin{array}{cc}
 \frac{-2+2 \sqrt{5}-i \sqrt{50-10 \sqrt{5}}+3 i \sqrt{10-2 \sqrt{5}}}{2
   \left(-3+\sqrt{5}\right)^2} & \frac{20-8 \sqrt{5}-3 i \sqrt{50-10 \sqrt{5}}+7 i
   \sqrt{10-2 \sqrt{5}}}{2 \left(-3+\sqrt{5}\right)^2} \\
 \frac{i \sqrt{10-2 \sqrt{5}}}{-3+\sqrt{5}} & \frac{8-4 \sqrt{5}-i \sqrt{50-10
   \sqrt{5}}+i \sqrt{10-2 \sqrt{5}}}{2 \left(-3+\sqrt{5}\right)^2} \\
\end{array}
\right),
\\&&
hn_{312} = e^{-4\pi i/5} h_{312},
\\&&
hn_{412}=g_{34} hn_{312} g_{34}^{-1},
\\&&
{\small
=\left(
\begin{array}{cc}
 \frac{8-4 \sqrt{5}-i \sqrt{50-10 \sqrt{5}}+i \sqrt{10-2 \sqrt{5}}}{2
   \left(-3+\sqrt{5}\right)^2} & -\frac{2 i \sqrt{10-2 \sqrt{5}}
   \left(-2+\sqrt{5}\right)}{\left(-3+\sqrt{5}\right)^2} \\
 -\frac{i \sqrt{10-2 \sqrt{5}}}{-3+\sqrt{5}} & \frac{8-4 \sqrt{5}+i \sqrt{50-10
   \sqrt{5}}-i \sqrt{10-2 \sqrt{5}}}{2 \left(-3+\sqrt{5}\right)^2} \\
\end{array}
\right)}.
 \end{eqnarray*}
 Then $h_{34}$ is the circuit matrix (i.e. the monodromy matrix) for the closed 
 arc obtained by the movement of $\lambda$ going around  $u=1$ in the positive sense 
 with respect to the basis $^t(\eta_2,\eta_3)$ as a left action.
 Analogously, $h_{45}$ is the circuit matrix coming from the closed arc around $u=\infty$, and $h_{412}$ is that around $u=0$.
 They are of order 5 as transformations in ${\rm Aut} (\mathcal{D})$.
 We use the symbol  $n$ for the normalized matrix that has the determinant =1. Set
\begin{eqnarray*}
&&
h_{\beta} = i \ \overline{\omega} \begin{pmatrix} -1&\omega \cr 1&1\end{pmatrix}.
\end{eqnarray*}
is an involutive element of  ${\rm Aut} (\mathcal{D})$ that exchanges $0$ and $\mathcal{F}({\infty}) =\omega (=\frac{1-\sqrt{5}}{2})$  and 
that fixes their middle point $P_{\beta}$ with $\det (h_{\beta} )=1$.
Set
\begin{eqnarray*}
&&
h_{Vck} = \begin{pmatrix} \overline{\omega} \rho_5&1 \cr  \overline{\omega}&\overline{\omega} \rho_5^{-1}\end{pmatrix}, \ (\rho_5= e^{2\pi i/5}).
\end{eqnarray*}

It causes a rotation of the  generating triangle $\nabla_v=\rho_5 ^{-1} \nabla (\xi_1,\xi_2,\xi_3) = 
\nabla (Cho_1, Cho_2,0)$ ,
where  $( Cho_1=\rho_5^{-1} \xi_1,Cho_2=\rho_5^{-1}\xi_2)$.
The center of gravity $Vc =\rho_5 ^{-1} u_2$ of the triangle $\nabla_v$ 
is the fixed point of  $h_{Vck} $ (see Fig. 5.1).

Hence,
$h_{34} h_{Vck}h_{34}^{-1}$ makes a rotation of $\mathcal{F(}1)=0,\mathcal{F}(0),\mathcal{F}(\infty)$ in this order,
and is a generator of the isotropy subgroup of  $V_c$ in the group $\Delta (3,3,5)$.
By the same way, $g_{34}h_{34} h_{Vck}h_{34}^{-1}g_{34}^{-1}$ makes a rotation of the reflection triangle $\rho_5^{-1}\nabla (\overline{\xi_1},\xi_2,\xi_3)$ and 
is a generator of the isotropy subgroup of its center of gravity $V_c'$, the reflection of $V_c$ relative to
the line $Cho_1, u_3$, in the group $\Delta (3,3,5)$.

As a cosequence, 
\begin{eqnarray} \label{eq: 335genesystem}
\begin{cases}
h_{34}, \cr 
rot_{V_c}=h_{34} h_{Vck}h_{34}^{-1},\cr
rot_{V_c'}=g_{34}h_{34} h_{Vck}h_{34}^{-1}g_{34}^{-1}
\end{cases}
\end{eqnarray}
 are the generators of the isotropy subgroups of $u_3=0, V_c, V_c'$ in  $\Delta (3,3,5)$, respectively. 
Then  (\ref{eq: 335genesystem}) is a generator system of $\Delta (3,3,5)$.

By a similar consideration $g_{34}, h_{\beta},h_{34} h_{Vck}h_{34}^{-1}$ is a generating system of the 
normalizer group $\Delta (2,3,10)$.
  
Set $\widetilde{\Delta (5,5,5)}$ be the shifted monodromy group  in ${\rm Aut} (\bm{H})$ by 
Remark \ref{DHtransformRem}. The system given by
 \begin{eqnarray*}
 &&
  \tilde{h}_{34}
  =\left(
\begin{array}{cc}
 \frac{1}{4}+\frac{\sqrt{5}}{4} & \frac{5}{4} \sqrt{\frac{1}{2}
   \left(1+\sqrt{5}\right)}-\frac{1}{4} \sqrt{\frac{5}{2} \left(1+\sqrt{5}\right)} \\
 \frac{1}{4} \sqrt{\frac{1}{2} \left(1+\sqrt{5}\right)}-\frac{1}{4} \sqrt{\frac{5}{2}
   \left(1+\sqrt{5}\right)} & \frac{1}{4}+\frac{\sqrt{5}}{4} \\
\end{array}
\right)
   \\&&
 {\small
 \widetilde{hn}_{45}= \left(
\begin{array}{cc}
 \frac{(-1)^{4/5} \left(1+2 (-1)^{2/5}-\sqrt{5}\right)}{-3+\sqrt{5}} & -\frac{\left(1+(-1)^{3/5}\right)
   \left(-1+\sqrt{5}\right)}{-3+\sqrt{5}} \\
 \frac{2 \left(-1+(-1)^{2/5}\right)}{-3+\sqrt{5}} & \frac{\sqrt[5]{-1} \left(-1+2 (-1)^{3/5}+\sqrt{5}\right)}{-3+\sqrt{5}} \\
\end{array}
\right)}
\\&&
 \widetilde{hn}_{412} =
 \\&&
  {\small
\left(
\begin{array}{cc}
 \frac{8+8 i \sqrt[4]{5}-4 \sqrt{5}-4 i 5^{3/4}+3 i \sqrt{-10+6 \sqrt{5}}-i
   \sqrt{-50+30 \sqrt{5}}}{2 \left(-3+\sqrt{5}\right)^2} & 0 \\
 \frac{-\sqrt{50-20 \sqrt{5}}+2 \sqrt{25-5 \sqrt{5}}+3 \sqrt{10-4 \sqrt{5}}-4
   \sqrt{5-\sqrt{5}}-\sqrt{-5+3 \sqrt{5}}+\sqrt{5 \left(-5+3 \sqrt{5}\right)}}{\sqrt{2}
   \sqrt[4]{5} \left(-3+\sqrt{5}\right)^2} & 0 \\
\end{array}
\right)
}
\\&&
+
  {\small
\left(
\begin{array}{cc}
0 & \frac{\sqrt[4]{5} \left(-8
   \sqrt[4]{5}+4\ 5^{3/4}-\sqrt{50-10 \sqrt{5}}+\sqrt{10-2 \sqrt{5}}+3 \sqrt{-10+6
   \sqrt{5}}-\sqrt{-50+30 \sqrt{5}}\right)}{\left(-3+\sqrt{5}\right)^2 \sqrt{2
   \left(-1+\sqrt{5}\right)}} \\
 0& \frac{-3 i \sqrt{10-4 \sqrt{5}}+i
   \left(\sqrt{50-20 \sqrt{5}}+2 \left(-2+\sqrt{5}\right) \left(\sqrt{5-\sqrt{5}}+i
   \sqrt{-1+\sqrt{5}}\right)\right)}{\left(-3+\sqrt{5}\right)^2 \sqrt{-1+\sqrt{5}}} \\
\end{array}
\right)
}
\end{eqnarray*}
generates $\widetilde{\Delta (5,5,5)}$. They are represented by the basis of $\bm{B}$:
 \begin{eqnarray*}
 &&
  \tilde{hn}_{34} = \frac{\overline{\omega}}{2} M_1+\frac{\omega}{2} M_x
 \\&&
  \tilde{hn}_{45} = \frac{1}{2} \overline{\omega} M_1+(-1+ \frac{1}{2} \overline{\omega})M_x+
  (-1+ \frac{1}{2} \overline{\omega}^3)M_y+(-\frac{1}{2} \overline{\omega} )M_z
\\&&
  \tilde{hn}_{412} = -\frac{1}{2} \overline{\omega}  M_1  + (1 - \frac{1}{2}\omega) M_x + M_z.
 \end{eqnarray*}
 So we have , by observing Table 2.1 of Takeuchi, 
 \begin{prop} 
 $\langle   \tilde{h}_{34} ,  \tilde{h}_{45} ,  \tilde{h}_{412} \rangle =\widetilde{\Delta (5,5,5)} \subset \Gamma^1(\bm{B}, \mathcal{O}\}$.
 \end{prop}
 Set $\widetilde{\Delta (3,3,5)} $ be the shifted unit group $\Gamma^{+} (\bm{B},\mathcal{O})$
 by
Remark \ref{DHtransformRem}.  
By observing (5.1), 
 \begin{prop}
 We have a system of generators of the unit group $\Gamma ^{+} (\bm{B},\mathcal{O})=\widetilde{\Delta (3,3,5)} $:
 \begin{eqnarray}
 \begin{cases}
 \widetilde{hn}_{34}, \cr
  \widetilde{rotn}_{V_c} = \widetilde{(h_{34}h_{Vck}h_{34}^{-1})}, \cr
   \widetilde{rotn}_{V_c'}=\widetilde{(g_{34}h_{34}h_{Vck}h_{34}^{-1}g_{34}^{-1})}.
 \end{cases}
 \end{eqnarray}
 They are represented in terms of the basis of $\mathcal{O}_{\bm{B}}$:
 \begin{eqnarray} \label{eq: 335geneBhyouji}
 \begin{cases}
 \widetilde{hn}_{34} =  \frac{\overline{\omega}}{2} M_1+\frac{\omega}{2} M_x= BG_3,\cr
   \widetilde{rotn}_{V_c} = \frac{1}{2} M_1+(-\frac{1}{2} +\frac{1}{2} \omega)M_x+(-\frac{1}{2} \omega) M_y+(-\frac{1}{2}) M_z
   =(BG_1,BG_2,BG_3,BG_4)\cdot (-\overline{\omega} ,1,1+\overline{\omega},-1) , \cr
   \widetilde{rotn}_{V_c'}= \frac{1}{2} M_1+(-\frac{1}{2} +\frac{1}{2} \omega)M_x+(\frac{1}{2} \omega) M_y+(-\frac{1}{2}) M_z
   = (BG_1,BG_2,BG_3,BG_4)\cdot (-1-\overline{\omega} ,1,1+\overline{\omega},0),
 \end{cases}
 \end{eqnarray}
 where $\cdot$ indicates the usual inner product of real vectors.
 \end{prop}

\subsection{Examples}
\paragraph{Example 5.1} 
For a CM field  $M=\bm{Q} (\sqrt{5}, \sqrt{-7})$ over $F=\bm{Q} (\sqrt{5})$, we have $h(M)=1$, ($h(M)$ is the class number of $M$).
By taking
\[
G_0=(-3+2\omega )BG_1+2\omega BG_2+(4-2\omega )BG_3+(-2\omega )BG_4,
\]
it holds ${\rm Tr} (G_0)=0, {\det} (G_0)=7$. So $G_0^2+7E=0$. By the correspondence $\sqrt{-7}\mapsto G_0$, we realize an embedding of $M$ into $\bm{B}$.
We have unique fixed point of $G_0$ in $\mathcal{D}$:
\[
u_0=(-0.205396\cdots) - (0.0667372 \cdots )i.
\]

By an approximate calculation, we see
$\tilde{\varphi }(u_0)= (8.3782124850378702032551165531909913589\cdots ) i$.
It holds
$\tilde{\varphi }(u_0)^2=-\frac{2527}{36} =-2^{-2}3^{-2}19^2\times 7$.
 Then, 
  $\tilde{\varphi }(u_0)\equiv \sqrt{-7} \ ({\Mod} {\bm{Q}^{\ast}})\in M=C(M)$.
 It means that  $\tilde{\varphi} (u)$ is the modular function that gives the canonical model of 
 $\bm{H}/\Gamma (\mathcal{O}_{\bm{B}},1)$ for (8), and that the normalized 
 modular function $\varphi (u)$ does not bring the canonical model.
\paragraph{Example 5.2} We find  in \cite{HHRWH} that a CM field $M=\bm{Q} \left( \sqrt{-(5+\sqrt{5})}\right)$  has $h(M)=2$.
Setting \[ G_0=(3-3 \omega )BG_1+BG_2+(-4+2 \omega )BG_3+(-1+\omega )BG_4, \]
we obtain
\[ G_0^2+5+\sqrt{5}=0.\] 
So $M=F(G_0)$ in $\bm{B}$. 
We have unique fixed point  of $G_0$ :
\[ u_0=-0.164894 - 0.119803 I\in \mathcal{D}.\]
We have
\[ \varphi (u_0)^2=-165.3749999999584 = -3^3 7^2/2^3 \]
(Note that $\tilde{\varphi} (u)=\sqrt{-3} \varphi (u)$).
So we obtain the Hilbert class field 
$C(M)=M(\sqrt{2})$.
  \paragraph{Example 5.3} Set  $M=\bm{Q} (\sqrt{-(65-26\sqrt{5})})$. Due to \cite{HHRWH}  and \cite{H-P} $h(M)=2$.
Take
  \[ G_0=(1 - 2 \omega) BG1 + 2 BG2 + (-8 - 2 \omega ) BG3 + (-2 \omega ) BG4).\]
 Then
\[ G_0^2+65-26\sqrt{5}=0.\] 
Hence $G_0$ is a generator of $M$ in $\bm{B}$. We have a fixed point  $u_0$ of $G_0$:
\begin{eqnarray*}
&&
 u_0=-0.2884031937082062430429292960544310724595352385781433875628704276940\cdots
\\&&
 +(0.2095371854415799547791501532228242020959121464954639149713389790753\cdots ) 
i \in \mathcal{D}
\end{eqnarray*}
Hence,
%
\[ \varphi (u_0)^2=-0.16717727965490681624739779831313980904 \cdots .\]
By the expansion into a continued fraction
\[
-\varphi (u_0)^2
=[0, 5, 1, 53, 1, 1, 3, 4, 1, 12, 7, 74, 2, 2, 41105985538320721741, \cdots]
\]
Hence,
\[
\varphi (u_0)^2=-13\cdot 29^2\cdot 79^2\cdot 2^{-8}\cdot 3^{-13}
\]
As a consequence we have the Hilbert class field $C(M)=M(\sqrt{13})$.

\paragraph{Example 5.4}  
  The case $M=\bm{Q} (\sqrt{5}, \sqrt{-23})$. 
 It holds $h(M)=3$ (note that $h(\bm{Q} (\sqrt{-23}))=3$ also).
We may choose two different generators of $M$ in $\mathcal{O}_{\bm{B}}$:
\begin{eqnarray*}
&&
\frac{1}{2} (M_1+(-3+\omega) M_x+(1-\omega) M_y-3\omega M_z) 
=-2 BG_1+3\omega BG_2+(4-3\omega) BG_3+(-1-\omega)BG_4
\\&&
\frac{1}{2} (M_1+(1-3\omega) M_x+(1-\omega) M_y-\omega M_z) 
=(4-3\omega) BG_1+\omega BG_2+(-4+\omega) BG_3-BG_4.
\end{eqnarray*}
Let $\tilde{\varphi} _4, \tilde{\varphi} _{11}$ be the values of $\tilde{\varphi} =\Phi (u)-\frac{1}{2}$ 
at the regular fixed point respectively. We have approximate values
\begin{eqnarray*}
&&
\tilde{\varphi }_4=0.41467884460813106945405483718570037432931049417786518217
\\&&
-0.99585830192876518343449922550065726567619225232766421487 i,
\\&&
\tilde{\varphi }_{11}=13.719509519930672493059974467190630795700234699440610387 i.
\end{eqnarray*}
According to our main theorem any of $\tilde{\varphi} _4, \overline{\tilde{\varphi} _4}, \tilde{\varphi} _{11}$ be a generator of
$C(M)$ over $F$.
Because $[C(M):M]=3$,
also any of $\tilde{\varphi} _4^2, \overline{\tilde{\varphi} _4}^2, \tilde{\varphi} _{11}^2$ be a generator of $C(M)$.
We have approximate values
\begin{eqnarray*}
&&
\check{r}_1=-(2^8 \tilde{\varphi }_4^2+2^8\overline{\tilde{\varphi }_4}^2+2^8 \tilde{\varphi }_{11}^2)=\frac{298002375630573376}{6131066257801},
\\&&
\check{r}_2=2^8 \tilde{\varphi }_{11}^2(2^8 \tilde{\varphi }_4^2+2^8\overline{\tilde{\varphi }_4}^2)+2^{16} \tilde{\varphi }_4^2 \overline{\tilde{\varphi }_4}^2
=\frac{27944558699379372032}{1375668606321},
\\&&
\check{r}_3=-2^{24} \tilde{\varphi }_{11}^2 \tilde{\varphi }_4^2 \overline{\tilde{\varphi }_4}^2=\frac{146663661576709210112}{34296447249}.
\end{eqnarray*}
Hence, we obtain a cubic equation for
 $2^8 \tilde{\varphi }_4^2, 2^8\overline{\tilde{\varphi }_4}^2, 2^8 \tilde{\varphi }_{11}$ with rational coefficients:
\begin{eqnarray} \label{eq: phiequation}
&&
t^3+\check{r}_1t^2+\check{r}_2t+\check{r}_3=0,
\end{eqnarray} 
where $t=2^8 \tilde{\varphi }^2$ . 
The above $(\ref{eq: phiequation})$ is a defining equation of the Hilbert class field of
 $M$. Putting
$Y=11^{3}2^{8}5^{4}\tilde{ \varphi}^{-1}$, we have integral defining equation
\begin{eqnarray} \label{eq:j0equation}
&&
Y^3+19268Y^2+12444768657 Y+6131066257801=0.
\end{eqnarray}
Setting $Y^3+r_1 Y^2+r_2Y+r_3=0$ for it, 
\begin{eqnarray*}
&&
r_3=19^{10},
\\&&
r_2=3^2 1382752073,
\\&&
r_1=2^2 4817,
\\&&
\mbox{the\ discriminant} \ {\rm dsc} =-19^{4}61^2 79^2 89^2 109^2 149^2 229^2 \times 23.
\end{eqnarray*}
So we see directly  (\ref{eq:j0equation}) defines a Galois extension over $M=F(\sqrt{-23})$ and non-Galois
over $\bm{Q}$.

\section{Appendix}
The list of arithmetic triangle groups by K. Takeuchi (1977) \cite{Tku1}:
{\center
~
\vskip 3mm
\begin{tabular}{ccccccc} \hline \hline
Class& $(e_1,e_2,e_3)$&$F$& Disc.&\\ \hline
I&$(2,3,\infty),(2,4,\infty),(2,6,\infty),(2,\infty ,\infty),(3,3,\infty),$&$\bm{Q}$&$(1)$ \\
&$(3,\infty ,\infty),(4,4,\infty),(6,6,\infty),(\infty,,\infty,\infty)$&& \\ \hline
II&$(2,4,6),(2,6,6),(3,4,4),(3,6,6)$&$\bm{Q}$&$(2)(3)$ \\ \hline
III&$(2,3,8),(2,4,8),(2,6,8),(2,8,8),(3,3,4)$&$\bm{Q}(\sqrt{2})$&$\frak{p}_2$ \\
&$(3,8,8),(4,4,4),(4,6,6),(4,8,8)$&& \\ \hline
IV&$(2,3,12),(2,6,12),(3,3,6)$&$\bm{Q}(\sqrt{3})$&$\frak{p}_2$ \\
&$(3,4,12),(3,12,12),(6,6,6)$&& \\ \hline
V&$(2,4,12),(2,12,12),(4,4,6),(6,12,12)$&$\bm{Q}(\sqrt{3})$&$\frak{p}_3$ \\ \hline
VI&$(2,4,5),(2,4,10),(2,5,5)$&$\bm{Q}(\sqrt{5})$&$\frak{p}_2$ \\
&$(2,10,10),(4,4,5),(5,10,10)$&& \\ \hline
VII&$(2,5,6),(3,5,5)$&$\bm{Q}(\sqrt{5})$&$\frak{p}_3$ \\ \hline
VIII&$(2,3,10),(2,5,10),(3,3,5),(5,5,5)$&$\bm{Q}(\sqrt{5})$&$\frak{p}_5$ \\ \hline
IX&$(3,4,6)$&$\bm{Q}(\sqrt{6})$&$\frak{p}_2$ \\ \hline
X&$(2,3,7),(2,3,14),(2,4,7),(2,7,7)$&$\bm{Q}(\cos (\pi /7))$&$(1)$ \\
&$(2,7,14),(3,3,7),(7,7,7)$&& \\ \hline
XI&$(2,3,9),(2,3,18),(2,9,18)$&$\bm{Q}(\cos (\pi /9))$&$(1)$ \\
&$(3,3,9),(3,6,18),(9,9,9)$&& \\ \hline
XII&$(2,4,18),(2,18,18),(4,4,9),(9,18,18)$&$\bm{Q}(\cos (\pi /9))$&$\frak{p}_2\frak{p}_3$ \\ \hline
XIII&$(2,3,16),(2,8,16),(3,3,8),(4,16,16),(8,8,8)$&$\bm{Q}(\cos (\pi /8))$&$\frak{p}_2$ \\ \hline

XIV&$(2,5,20),(5,5,10)$&$\bm{Q}(\cos (\pi /10))$&$\frak{p}_2$ \\ \hline

XV&$(2,3,24),(2,12,24),(3,3,12)$&$\bm{Q}(\cos (\pi /12))$&$\frak{p}_2$ \\
&$(3,8,24),(6,24,24),(12,12,12)$&& \\ \hline

XVI&$(2,5,30),(5,5,15)$&$\bm{Q}(\cos (\pi /15))$&$\frak{p}_3$ \\ \hline
XVII&$(2,3,30),(2,15,30),(3,3,15),(3,10,30),(15,15,15)$&$\bm{Q}(\cos (\pi /15))$&$\frak{p}_5$ \\ \hline
XVIII&$(2,5,8),(4,5,5)$&$\bm{Q}(\sqrt{2},\sqrt{5})$&$\frak{p}_2$ \\ \hline
XIX&$(2,3,11)$&$\bm{Q}(\cos (\pi /11))$&$(1)$ \\ \hline
\hline
\end{tabular}

\vskip 5mm
\centerline{Table App.1: Total list of arithmetic triangle groups }
}

\paragraph{[Aknowledgement] } We are grateful to the referee for the deep insight and kind comments for our manuscript. Especially, 
he (or she) carefully checked our main theorem, and he (or she) kindly suggested to exclude the two quadrangle unit groups 
from our consideration.

\noindent
Atsuhira Nagano (Waseda univ.)
e-mail: atsuhira.nagano@gmail.com
\par \vskip 4mm
\noindent
Hironori Shiga (Chiba univ., Corresponding Author)
    e-mail: shiga@math.s.chiba-u.ac.jp


\par \vskip 5mm
This research was supported by 
$\cdot$) Waseda University Grant for Special Research Project (2015B-191),
$\cdot$) The Sumitomo Foundation Grant for Basic Science Research Projects (150108), 
$\cdot$) The JSPS Program for Advancing Strategic International Networks 
to Accelerate the Circulation of Talented Researchers 
``Mathematical Science of Symmetry, Topology and Moduli,
Evolution of International Research Network based on OCAMI''
and
$\cdot$) Grant-in-Aid for Scientific Research (No. 15K04807), Japan Society for the Promotion of Science.

\end{document}